\documentclass[a4paper]{amsart}
\usepackage{amsmath,amssymb,amscd,bbm,dsfont,enumerate}
\usepackage{psfrag}
\usepackage{graphicx}
\usepackage{mathtools}
\newcommand{\N}{\mathbb N}  
\newcommand{\R}{\mathbb R}  
 
\renewcommand{\P}{\mathbbm P} 
 
\newcommand{\cf}{\mathbbm{1}} 
 
\newcommand{\trp}[1]{#1^{\mathrm{t}}} 
 
\newcommand{\obsf}{{h}} 

\newcommand{\beq}[1]{\begin{equation}\label{#1}}
\newcommand{\eeq}{\end{equation}}
\newcommand{\beqn}[1]{\begin{equation}\nonumber}
\newcounter{itemfreeze}
\newcommand{\dd}{\mathrm d}
\newcommand{\DD}{\mathrm D}
\newcommand{\abs}[1]{\lvert #1 \rvert}
\newcommand{\norm}[1]{\lVert #1 \rVert}
\newcommand{\var}[1]{\left[ #1 \right]}
\newtheorem{hypothesis}{Hypothesis} 
\newtheorem{theorem}{Theorem}
\newtheorem{lemma}{Lemma}
\theoremstyle{definition}
\newtheorem{definition}{Definition}
\renewcommand{\theenumi}{\alph{enumi}}
\renewcommand{\div}{\operatorname{div}}
\title[Variational Data Assimilation in Continuous Time]{Existence and Uniqueness For Variational Data Assimilation in Continuous Time}
\author{Jochen Br\"{o}cker}
\address{School of Mathematical, Physical, and Computational Sciences, University of Reading, Reading RG6 6AX, United Kingdom}
\email{j.broecker@reading.ac.uk} 
\thanks{The author was supported by the UK Engineering and Physical Sciences Research Council under grant agreement EP/L012669/1.
Fruitful discussions with Horatio Boedihardjo, Tobias Kuna, and Dan Crisan are gratefully acknowledged.
Suggestions by an anonymous referee helped to improve this work.}
\subjclass[2010]{Primary 49J55, 49K35; Secondary 86A10, 93E99, 60G35}
\begin{document}
\maketitle  
\begin{abstract}	
A variant of the optimal control problem is considered which is nonstandard in that the performance index contains ``stochastic'' integrals, that is, integrals against very irregular functions.
The motivation for considering such performance indices comes from dynamical estimation problems where observed time series need to be ``fitted'' with trajectories of dynamical models.
The observations may be contaminated with white noise, which gives rise to the nonstandard performance indices.
Problems of this kind appear in engineering, physics, and the geosciences where this is referred to as data assimilation.
Pathwise existence of minimisers is obtained, along with a maximum principle as well as preliminary results in dynamic programming. 
The results extend previous results on the maximum aposteriori estimator of trajectories of diffusion processes.
To obtain these results, classical concepts from optimal control need to be  substantially modified due to the nonstandard nature of the performance index, as well the fact that typical models in the geosciences do not satisfy linear growth nor monotonicity conditions.
\end{abstract}
%
%
%
%
\section{Introduction and main results}
\label{sec:introduction}
In many branches of science, particularly in physics, the geosciences, and engineering, there appears the problem of ``fitting'' observed time series with trajectories of dynamical models; typically these are ordinary or partial differential equations.
The aim might be to identify appropriate models or to estimate model parameters, but there is also considerable interest in the estimated trajectories themselves.
These problems are the main motivation for the results in this paper.
In the geosciences, these and related problems are referred to as data assimilation, and we will use this term here, although there has been relevant research in other communities which often predates research in the geosciences.
We refer to~\cite{derber89,tremolet06,kalnay01,evensen_enkf_2007} and references for an overview over data assimilation from a geoscientist's perspective; \cite{JAZ,sage68,sontag98} discuss similar problems in engineering; \cite{nakamura_inverse_modeling_2015} contains many examples of data assimilation problems in biology and other applied sciences.
In the geosciences, the estimation of plausible model trajectories which fit the observational record is particularly important, for at least two reasons.
Firstly, trajectories of atmospheric models that fit observations over a long temporal window might hold clues about past weather phenomena which were not directly observed. 
Examples are large and almost stationary pressure systems over the oceans which lead to very persistent weather patterns over the continents (blocking events) but which are not captured by historic instrumentation and thus appear only indirectly in historic weather records.
Secondly, the endpoint of such trajectories might be interpreted as a good guess of the then current state of the atmosphere and hence can be used as initial condition for predictions into the future. 
In the context of data assimilation, ``fitting'' can mean different things but often it involves the optimisation of some sort of error criterion or performance index, integrated over time, in which case data assimilation essentially becomes a variational problem.
If the observations are considered as a stochastic process, the performance index and also the fitted solutions might have a probabilistic interpretation.
At the same time though, observations which contain white noise give rise, as we will see, to nonstandard variational problems, because the performance index will then contain a stochastic integral.
To define the class of problems we want to investigate, we consider an interval $I = [0, T]$ with $T > 0$ and a finite dimensional vector space $E$ with norm $\abs{.}$ and dual $E'$.
Let $f:I\times E \to E$, $g: I\times E \to L(E,E)$ be two functions and consider the controlled initial value problem or {\em state equation}
\beq{equ:initvalproblem}
\dot{x}(t) = f(t, x(t)) + g(t, x(t)) u(t), 
\qquad t \in I,
\eeq 
with initial condition $\xi \in E$ and control function $u: I \to E$.
The function $f$ typically represents the physically relevant part of the model while $g$ might be needed to scale the control; in geophysical applications, $g$ is often the unit matrix.
A {\em process} with respect to $\xi$ is a pair $(x, u)$ so that $u: I \to E$ is measurable, $x: I \to E$ is absolutely continuous, $x(0) = \xi$, and the state equation~\eqref{equ:initvalproblem} is satisfied for almost every $t \in I$.
(At this point, it might not seem obvious why we need the mapping $g$ at all since a new control $\tilde{u}$ could be defined through $\tilde{u}(t) := g(t, x(t)) u(t)$; we will see below though that the performance index contains $u$ which cannot be replaced by $\tilde{u}$ though without further conditions.)
We aim to find controls $u$ so that, roughly speaking, $u$ is not ``too large'' so that $x$ is not ``too far'' from being a solution of the physically relevant part of the model, that is the state equation~\eqref{equ:initvalproblem} but with the control term $g(t, x(t)) u(t)$ omitted.
But on the other hand, we want $x$ to reproduce the observations, in the following sense: for some mapping $\obsf:I \times E \to \R^d$ (which is part of the problem statement), we aim to find controls $u$ so that $\obsf(t, x(t))$ is not ``too far'' from the observation $y(t)$ for all $t \in I$.
This problem might be approached by optimising an appropriate performance index with respect to the control which takes these two aims into account.
In this paper, we assume the observations to have the following structure.
Let $\{W(t), t \in I\}$ be the standard $d$--dimensional Wiener process on some probability space $(\Omega, \mathcal{B}, \P)$.
We can assume without loss of generality that $\Omega = C(I, \R^d)$, the space of continuous functions on $I$ with values in $\R^d$, $\mathcal{B}$ the Borel sigma algebra generated by the supremum norm topology, $\P$ the standard $d$--dimensional Wiener measure, and $W(t): \Omega \to \R^d; W(t)(\omega) = \omega(t)$ for each $t \in I$.
Heuristically, we define our observations as $y(t) = \dot{\zeta}(t) + r(t)$ where $\dot{\zeta}$, the desired signal, is Lebesgue integrable and $r$ is white noise.
In our analysis though we will work with the process $\eta(t) = \zeta(t) + W(t)$ for all $t \in I$, where $\zeta$ is a (probably random) absolutely continuous function.
Formally, $\eta$ can be considered as the observations in ``integrated form'', that is $\eta(t) = $``$\int_0^t y(t) \: \dd s $''.
We stress however that this connection is really only formal as $\eta$ will not have a classical derivative due to the presence of the Wiener process. 
In terms of performance indices, quadratic functionals of the form
\beq{equ:exampleerror}
\frac{1}{2} \int_I 
    \trp{\obsf(t, x(t))} R(t)\obsf(t, x(t)) \: \dd t
- \int_I \trp{\obsf(t, x(t))} R(t) \: \dd \eta(t)
+ \frac{1}{2} \int_I \trp{u(t)}S(t)u(t) \: \dd t
\eeq
have enjoyed popularity, where $R$ and $S$ are suitable functions with values in the nonnegative definite matrices.
The approach is then called weakly constrained 4--dimensional variational assimilation (4D--VAR) in the geosciences, see~\cite{derber89,tremolet06,kalnay01,evensen_enkf_2007}; this is not an exhaustive list of references, and most authors use a discrete time framework.
In the engineering community, the approach is known as minimum energy estimator~\cite{mortensen_likelihood_filtering_1968,hijab_asymptotic_bayesian_1984,krener_minimum_energy_2003}, see also~\cite{rogers_least_action_2013}.
The approach has been interpreted as maximum aposteriori (MAP) estimation of diffusion trajectories (see e.g.~\cite{JAZ}), but this cannot be justified rigorously. 
The correct interpretation of the functional~\eqref{equ:exampleerror} in terms of large deviations has been given by~\cite{hijab_asymptotic_bayesian_1984}.
In~\cite{zeitouni_maximum_1987,zeitouni_existence_1988}, it was shown that the MAP estimator is correctly defined as a minimiser of the Onsager--Machlup functional, which comprises the functional~\eqref{equ:exampleerror} but with further terms added.
A variant of the functional~\eqref{equ:exampleerror} is often found where instead of the first two terms the expression
\beqn{equ:exampleerror2}
\frac{1}{2} \int_I \trp{\big\{y(t) - \obsf(t, x(t))\big\}} 
    R(t) \big\{y(t) - \obsf(t, x(t))\big\} \: \dd t
\eeq
is used.
This form is fine if the function $y$ has the appropriate integrability properties (see e.g.~\cite{krener_minimum_energy_2003} for such a case) but is not well defined if $y$ contains white noise components, as is the case here.
The performance indices we will consider in this paper encompass the minimum energy functional as well as the Onsager--Machlup functional (as in~\cite{zeitouni_existence_1988}); given two functions $\phi:I \times E \times E \to \R$ and $\psi:I \times E \to L(\R^d, \R)$, consider the {\em cost functional}
\beq{equ:costfunctional}
A(x, u) = \int_I \phi(t, x(t), u(t)) \: \dd t + \int_I \psi(t, x(t)) \: \dd \eta(t)
\eeq 
with any process $(x, u)$ so that $t \to \phi(t, x(t), u(t))$ is integrable and $t \to \psi(t, x(t))$ has finite $p$-variation for some $p < 2$ (this ensures that the second integral in~\eqref{equ:costfunctional} is defined as a Young integral, as we will see later).
The functions $\phi$ and $\psi$ will be called the {\em deterministic } and the {\em stochastic running costs}, respectively, and the first and second integral in the cost functional~\eqref{equ:costfunctional} we will call the {\em deterministic } and the {\em stochastic costs}, respectively.
We will properly define the $p$--variation and the Young integral in Section~\ref{sec:existence} and summarise a few properties central to our analysis in Lemmas~\ref{lem:youngproperties}-\ref{lem:varwiener}.
In particular, we obtain that for any given observation path $\eta$, the integral $\int_I \psi(t, x(t)) \: \dd \eta(t)$ is well defined and finite whenever $x$ is a solution of the state equation~\eqref{equ:initvalproblem}.
In particular, attempts to calculate this integral using Stratonovi\v{c} or It\^{o} style partitions will give the same result.
In view of this, we can forget about the stochastic character of the observations and instead approach data assimilation pathwise for each realisation of the observations.
The final ingredient we add to our problem is a {\em control set} $U \subset E$, interpreted as the set of permitted values for our control function $u$.
A control function $u:I \to E$ will be called {\em feasible} with respect to $\xi$ if $u$ is measurable, $u(t) \in U$ for almost all $t \in I$ and there is a function $x$ so that $(x, u)$ is a process with respect to $\xi$.
In this case, the pair $(x, u)$ will be called a {\em feasible process} with respect to $\xi$.
A feasible process $(x, u)$ with respect to $\xi$ will be called an {\em admissible process} with respect to $\xi$ if it has finite costs; a control $u$ which is part of an admissible process will be called an {\em admissible control} with respect to $\xi$.
The qualifier ``with respect to $\xi$'' might be omitted if clear from the context.
We are now ready to formulate our problem statement
\par \vspace{1ex} \noindent \textbf{Problem VAR} (Variational Data Assimilation)
%
%
\itshape
Minimise the cost functional
\beqn{equ:objectivefunction}
A(x, u) 
= \int_I \phi(t, x(t), u(t)) \: \dd t 
+ \int_I \psi(t, x(t)) \: \dd \eta(t)
\eeq
subject to 
\begin{align}
\label{equ:constraintode}
\dot{x}(t) & = f(t, x(t)) + g(t, x(t))u(t) 
\qquad \mbox{for a.a.\ }t \in I\\
\label{equ:constraintu}
u(t) & \in U  
\qquad \mbox{for a.a.\ }t \in I\\
\label{equ:constraintx}
x(0) & = \xi,
\end{align}
that is, over all admissible pairs with respect to $\xi$. 
\vspace{1ex}
\normalfont
The remainder of this section will be devoted to presenting our main hypotheses and results.
For functions on $I$, we use the norms
\beqn{equ:normdef}
\norm{u}_r := \left( \int_I |u(s)|^r \; \dd s \right)^{1/r}
\qquad \mbox{and} \qquad
\norm{u}_{\infty} := \sup_{s \in I} |u(s)|
\eeq 
For $r > 1$, we define $\mathcal{U}_r$ as the set of all measurable control functions $u:I \to E$ with $u(t) \in U$ almost surely and $\norm{u}_r < \infty$.
\begin{hypothesis}
\label{hyp:hypothesisI}
\begin{enumerate}
\item \label{hyp:thexcont} $f, g$ are continuous. 
\item \label{hyp:thexhoel} For all $R \geq 0$ the function $\psi$ is H\"{o}lder on $I \times \{x \in E; \abs{x} \leq R\}$ with constant $K_R$ and exponent $\kappa > \frac{1}{2}$.
Further
\beq{hyp:thexhder}
    K_R \leq C_{\psi} (1 + R^{\alpha})
\eeq
for some nonegative constants $C_{\psi}, \alpha$.
\item \label{hyp:thexcontrolset} $U$ is closed and convex.
\item \label{hyp:thexconvexity} The function $\phi$ is continuous on $I \times E \times U$, convex in $u$ and satisfies a lower bound of the form $\phi(t, x, u) \geq C_{\phi} (\abs{u}^r - \abs{x}^{\delta})$ for constants $C_{\phi} > 0$, $\delta > 0$ and $r > 1$.
\item \label{hyp:thexvarbounds} There exists $\xi \in E$  and nonegative constants $C_s, C_e, \beta, \gamma$ so that for any control $u \in \mathcal{U}_r$ and any local solution $x$ for the state equation~\eqref{equ:initvalproblem} with control $u$ and initial condition $\xi$ we have
\begin{align}
\norm{\dot{x}}_{r} 
    & \leq C_s (1 + \norm{u}_r^{\beta})
        \label{hyp:thexnonlin}\\
\norm{x}_{\infty} 
    & \leq C_e (1 + \norm{u}_r^{\gamma}).
        \label{hyp:thexenergy}
\end{align}
Further, $\gamma (\alpha + \kappa - \frac{1}{2}) + \frac{\beta}{2} < r$, and $\gamma \delta < r$.
\end{enumerate}
\end{hypothesis}
Hypothesis~\ref{hyp:hypothesisI}\eqref{hyp:thexhoel} represents growth and regularity conditions on the stochastic costs; the estimate~(\ref{hyp:thexnonlin}) quantifies the nonlinearity in the state equation, while estimate~(\ref{hyp:thexenergy}) might represent an energy estimate.
By local solution in Hypothesis~\ref{hyp:hypothesisI}\eqref{hyp:thexvarbounds} we mean that $x$ is a solution of the state equation~\eqref{equ:initvalproblem} but possibly only over a  smaller interval $[0, T_1]$ with $T_1 \leq  T$.
Note that we can assume without loss of generality that $\gamma \leq \beta$, since the estimate~(\ref{hyp:thexnonlin}) always implies an energy estimate~(\ref{hyp:thexenergy}) with $\gamma = \beta$.
We stress that these conditions might be satisfied even if $f$ fails to exhibit linear growth. 
To illustrate this, we present a simple example.
This example belongs to a wider class of problems, to be discussed in Section~\ref{sec:motivatingexample} and motivated by data assimilation in geophysical fluid dynamics, in which the state equation has a characterisic quadratic nonlinearity, so $\beta = 2$, but $\gamma = 1$ due to energy conservation, and this will turn out to be crucial.
Our example is the Lorenz'63 system in the form given in~\cite{temam_infinite_dynamics_1997}.
The vector field has the form $f = f_1 + f_2$, where
\beq{equ:lorenz}
f_1(x, y, z) 
 = \left( 
\begin{array}{c}
-\sigma x + \sigma y\\
-\sigma x - y \\
-bz - b(r + \sigma) 
\end{array}
\right)
\qquad \mbox{and} \qquad
f_2(x, y, z) 
 = \left( 
\begin{array}{c}
0\\
-xz\\
xy 
\end{array}
\right),
\eeq 
with $\sigma, r, b$ positive parameters (see~\cite{temam_infinite_dynamics_1997} for their interpretation).
Typical values for these parameters are $\sigma = 10, r = 28$, and $b = 8/3$.
Note that $f_2$ is quadratic while $f_1$ is linear and stable. 
We complement this by setting $g = \cf$ and putting no further constraints on the control, that is we set $U = \R^3$.
Further, we consider a cost functional of the form~\eqref{equ:exampleerror} with $h(x, y, z) = x, R = 1$, and $S = \cf$ so that $\phi(x, y, z, u) = \frac{1}{2} x^2 + \frac{1}{2} u^2$ and $\psi(x, y, z) = -x$.
With regards to Hypothesis~\ref{hyp:hypothesisI}, only item~\eqref{hyp:thexvarbounds} is not obvious.
However, since $(x, y, z)^t f_2(x, y, z) = 0$ and $f_1$ is stable, the energy estimate~\eqref{hyp:thexenergy} in Hypothesis~\ref{hyp:hypothesisI}\eqref{hyp:thexvarbounds} follows with $\gamma = 1$ from an application of the Bellmann--Gr\"{o}nwall Lemma (see \cite{temam_infinite_dynamics_1997} and Sec.~\ref{sec:motivatingexample} for more details).
Using this bound on $\|x\|_{\infty}$ directly in the state equation, we obtain the nonlinearity estimate~\eqref{hyp:thexnonlin} in Hypothesis~\ref{hyp:hypothesisI}\eqref{hyp:thexvarbounds} with $\beta = 2$, due to $f$ being quadratic in leading order.
As another note on Hypothesis~\ref{hyp:hypothesisI}, it might seem at first sight that Hypothesis~\ref{hyp:hypothesisI}\eqref{hyp:thexhoel} as well as the lower bound on~$r$ in Hypothesis~\ref{hyp:hypothesisI}\eqref{hyp:thexvarbounds} both become {\em less} restrictive for smaller $\kappa$.
It has to be kept in mind though that a smaller H\"{o}lder exponent $\kappa$ in general implies a larger H\"{o}lder constant $K_R$ and, in our context, a larger $\alpha$ in the estimate~\eqref{hyp:thexhder}. 
Our first result ensures that under Hypothesis~\ref{hyp:hypothesisI}, Problem~VAR is well defined.
\begin{lemma}
Under Hypothesis~\ref{hyp:hypothesisI}, the following holds:
\label{lem:welldefined}
\begin{enumerate}
\item For any $u \in \mathcal{U}_r$, the state equation has a (not necessarily unique) solution $x$ with $x(0) = \xi$.
\item If $(x, u)$ is feasible, the stochastic part of the cost functional is well defined and finite.
\item If $(x, u)$ is feasible, the function $t \to \phi(t, x(t), u(t))$ is measurable and bounded below by an integrable function; hence the deterministic costs are well defined as an element of $\R \cup \{\infty\}$.
\item If $(x, u)$ is admissible, then $\norm{u}_r < \infty$.
\end{enumerate}
\end{lemma}
In view of Lemma~\ref{lem:welldefined}(a), every element $u \in \mathcal{U}_r$ is feasible with respect to $\xi$, and by item~(d), every admissible $u$ must lie in $\mathcal{U}_r$.
By items~(b, c), a feasible process is admissible if and only if the deterministic costs are finite. 
In particular, admissibility does not depend on the particular observation path.
%
The following is our main existence result:
\begin{theorem}
\label{thm:existence}
Suppose that Hypothesis~\ref{hyp:hypothesisI} is in force and that there is an admissible process $(x_0, u_0)$ with repect to $\xi$ specified in Hypothesis~\ref{hyp:hypothesisI}.
Then for any observation path $\{\eta_t, t \in I\}$, Problem~VAR admits a solution $(x_*, u_*)$.
Further, $\norm{u_*}_r < \infty$. 
\end{theorem}
Both Lemma~\ref{lem:welldefined} and Theorem~\ref{thm:existence} will be proved in Section~\ref{sec:existence}.
In Section~\ref{sec:maximumprinciple}, we will prove a maximum principle (Theorem~\ref{thm:maximumprinciple}).
For the maximum principle, we strengthen our Hypothesis:
\begin{hypothesis}
\label{hyp:hypothesisII}
Hypothesis~\ref{hyp:hypothesisI} and
\begin{enumerate}
\item \label{hyp:thmaxstate} $f, g$ have continuous partial derivatives with respect to the second variable. 
\item \label{hyp:thmaxdetcost} $\phi$ has a continuous partial derivative with respect to the second variable. 
\item \label{hyp:thmaxinteg} 
There exists constants $R > 0$, $c \geq  0$ and an integrable function $d: I \to \R_{\geq 0}$ so that whenever $(x, u)$ is an admissible process and $\abs{y - x(t)} \leq R$, then $\abs{\DD_2 \phi(t, y, u(t))} \leq c \abs{\phi(t, y, u(t))} + d(t)$ almost surely. 
\item \label{hyp:thmaxhoel} $\psi$ has a partial derivative with respect to the second variable which is locally H\"{o}lder in $(t, x)$ with exponent larger than $\frac{1}{2}$. 
\end{enumerate}
\end{hypothesis}
\begin{theorem}[Maximum principle]
\label{thm:maximumprinciple}
Assume Hypothesis~\ref{hyp:hypothesisII} for a certain $r > 1$, and initial condition $\xi$.
Fix $\eta$ and let $(x, u)$ be a global minimiser for Problem~VAR.
Then for almost all $t \in I$ the control $u$ satisfies the condition
\beq{equ:maxcondition}
H(t, x(t), \lambda(t), u(t)) 
= \inf_{v \in U} H(t, x(t), \lambda(t), v)
\eeq
where the function $\lambda: I \to E'$ is the unique solution to the integral equation
\begin{multline}
\label{equ:costate}
\lambda(t) = \int_{t}^{T} 
    \lambda(s) \big\{
        \DD_2 f(s, x(s)) + \DD_2 g(s, x(s)) u(s)
    \big\} \: \dd s\\
    + \int_{t}^{T} \DD_2 \phi(s, x(s), u(s)) \: \dd s
    + \int_{t}^{T} \DD_2 \psi(s, x(s)) \: \dd \eta_s
\end{multline}
and 
\beqn{equ:hamiltonian}
H(t, x, \lambda, v) = \phi(t, x, \lambda, v) + \lambda \big( f(t, x) + g(t, x) v \big)
\eeq
\end{theorem}
The function $\lambda$ is referred to as the {\em costate}, and $E'$ denotes the dual space of $E$.
The third integral in Equation~\eqref{equ:costate} is to be interpreted as the element of $E'$ given by the linear function
\[
E \to \R; z \to \int_{t}^{T} \DD_2 \psi(s, x(s))z \: \dd \eta_s.
\]
If $(x, u)$ is an optimal pair for problem~VAR with respect to initial condition~$\xi$ and if $\lambda$ is a corresponding costate (in the sense of Theorem~\ref{thm:maximumprinciple}), then we will refer to $(x, u, \lambda)$ as an {\em optimal triple} with respect to initial condition~$\xi$.
Note that Hypothesis~\ref{hyp:hypothesisII}\eqref{hyp:thmaxstate} implies uniqueness to solutions of the state equation.
The condition~\eqref{hyp:thmaxinteg} in Hypothesis~\ref{hyp:hypothesisII} ensures that the second integral in the definition of $\lambda$ is well defined.
In fact, for the maximum principle it would be sufficient to know that this condition holds along optimal processes $(x, u)$.
In Sections~\ref{sec:regularitycontrols} and \ref{sec:uniqueness}, two applications of the maximum principle will be discussed. 
A precise formulation of the results will be given in those sections. 
The first result, Theorem~\ref{thm:controlregularity}, shows that under the conditions of the maximum principle and if $\phi(t, x, u)$ is strongly convex in $u$, an optimal control has finite $p$--variation for any $p > 2$.
It turns out that the regularity of the controls is limited by the regularity of the observations, which cannot be improved beyond $p$--variation with $p > 2$.
In particular, we do not expect optimal controls to be Lipschitz in general, as is the case in other optimal control problems.
The second application (in Sec.~\ref{sec:uniqueness}) concerns the value function and its relation to uniqueness of optimal controls.
The costs of any feasible control can be regarded as a function $J(\xi, u)$ on $E \times \mathcal{U}_r$, and the {\em value function} is defined as 
\[
V(\xi) = \inf_{u \in \mathcal{U}_r} J(\xi, u)
\]
Theorem~\ref{thm:uniqueness} then shows that an optimal control $u$ is unique for Problem~VAR if and only if the value function has a derivative at $\xi$.
Finally, section~\ref{sec:motivatingexample} discusses an application to a typical problem in geophysical data assimilation.
For this example, both a minimum energy cost function (known as weakly constrained 4d-VAR costfunction in the atmospheric sciences) as well as an Onsager--Machlup type cost function will be considered.
For the latter case, we obtain the existence of minimisers under substantially weaker assumptions than in~\cite{zeitouni_existence_1988}.
It is worth mentioning that Problem~VAR can be approached using classical methods of optimal control.
Under Hypothesis~\ref{hyp:hypothesisI} and if $\psi$ has a continuous derivative with respect to $(t, x)$, the stochastic costs can be integrated by parts and one obtains a standard control problem with running costs given by
\beq{equ:standardrunning}
\tilde{\phi}(t, x, u) 
 = \phi(t, x, u) 
 - \big\{ \DD_1 \psi(t, x) 
    + \DD_2 \psi(t, x) \left(f(t, x) + g(t, x) u\right)\big\}
    \eta(t).
\eeq
This approach has been taken in~\cite{zeitouni_existence_1988,zeitouni_change_of_variables_1990}, but it seems that stronger assumptions are required to make this approach work (in addition to the existence of the derivatives of $\psi$).
After all, some coercivity property of $\tilde{\phi}$ is required to obtain weakly compact level sets, and even if Hypothesis~\ref{hyp:hypothesisI}\eqref{hyp:thexconvexity} is imposed, further assumptions such as linear growth of $f$ are needed to make sure that for large $u$, the second term in Equation~\eqref{equ:standardrunning} does not override the first one.
Linear growth of $f$ however is not needed in our approach (and is in fact not satisfied for typical geophysical models).
Our approach instead relies on a more careful analysis of the stochastic costs, using the theory of $p$--variation and Young integrals.
%
%
\section{Proof of existence result}
\label{sec:existence}
Before proving Lemma~\ref{lem:welldefined} and the existence result (Thm.~\ref{thm:existence}), we will discuss the notions of $p$--variation and Young integrals.
We will not give proofs as these are standard facts or easy modifications thereof, and references will be provided below.
%
%
%
\begin{definition}
Let $I \subset \R$ be an interval and $V, W$ be finite dimensional vector spaces, with $\abs{.}$ denoting a generic norm.
\begin{enumerate}
\item Let $x \in C(I, V)$ and $p \geq 1$. 
Then the $p$--{\em variation} of $x$ is defined as
\beqn{equ:pvardef}
\var{x}_p  := \left( \sup \sum_i \abs{x(t_{i+1}) - x(t_{i})}^p \right)^{1/p}
\eeq
where the $\sup$ is taken over all finite dissections $D = \{0 = t_0 < t_1 < \ldots < t_n = T\}$ of I.
 
\item Let $x \in C(I, L(V, W))$ and $y \in C(I, V)$. 
The Young integral of $x$ against $y$ is given by
\beqn{equ:youngdef}
\int_I x(s) \: \dd y(s) := \lim \sum_i x(\tau_{i}) (y(t_{i+1})- y(t_{i}))
\eeq
where the sum runs over a finite dissection of $I$ and $\tau_{i} \in [t_{i}, t_{i+1}]$ for all $i$, and the limit is with respect to the resolution $\abs{D} := \sup_i \abs{t_{i+1} - t_{i}}$ of the dissections going to zero.
It is required that the existence of the limit and its value does not dependent on the choice of the dissections nor on the precise placement of the $\tau_{i}$.
\end{enumerate}
\end{definition}
\begin{lemma}
\label{lem:youngproperties} 
Let $x \in C(I, L(V, W))$ and $y \in C(I, V)$, both $\var{x}_p$ and $\var{y}_q$ finite with $\theta := \frac{1}{p} + \frac{1}{q} > 1$, then
\begin{enumerate}
\item the Young integral of $x$ against $y$ exists and 
\beqn{equ:youngproperties}
\abs{\int_I x(s) \: \dd y(s) - x(0) (y(T)- y(0)) }
\leq \frac{1}{1 - 2^{1 - \theta}} \var{x}_p \var{y}_q,
\eeq
\item the Young integral permits integration by parts:
\beqn{equ:young_by_parts}
\int_I x(s) \: \dd y(s) 
= x(T)y(T) - x(0)y(0) - \int_I y(s) \: \dd x(s).
\eeq
\end{enumerate}
\end{lemma}
For the first and second part of this Lemma, see~\cite{friz_rough_paths_2009}, Theorem~6.8 and Exercise~6.14, respectively.
In order to make sense of the integration by parts formula, we note that every element $y$ of $V$ gives rise to a linear mapping $\hat{y}:L(V, W) \to W, A \to A y$, and the integral on the right hand side of the integration by parts formula is to be understood as $\int_I \hat{y}(s) \: \dd x(s)$. 
With slight abuse of notation however, we will continue to write $y$ instead of $\hat{y}$.
\begin{lemma}
\label{lem:varproperties} 
Let $x \in C(I, L(V, W))$ and $y \in C(I, V)$
\begin{enumerate}
\item Suppose $y$ is absolutely continuous, then $\var{y}_p \leq \norm{\dot{y}}_1$ for all $p \geq 1$.
\item If $p > q$, then $\var{y}_p \leq \var{y}_q$.
\item (Interpolation inequality) If $p > q$, then 
\beqn{equ:interpolation}
\var{y}_p \leq \left(\var{y}_q\right)^{\frac{q}{p}} \left(\sup_{t,s \in I} \abs{y(t) - y(s)}\right)^{1 - \frac{q}{p}}.
\eeq
\item (Product rule)
\beqn{equ:productrule}
\var{x y}_p \leq \var{x}_p \norm{y}_{\infty} + \var{y}_p \norm{x}_{\infty}
\eeq
\item (Chain rule) Suppose that the function $\psi:I \times V \to W$ is H\"{o}lder on $I \times \{x \in V; \abs{x} \leq \norm{y}_{\infty}\}$ with exponent $\kappa$ and constant $K$, and that $\kappa p \geq 1$, then 
\beqn{equ:chainrule}
\var{\psi(., y(.))}_p \leq K (\var{y}_{\kappa p}^{\kappa} + \abs{I}^{\kappa})
\eeq
\end{enumerate}
\end{lemma}
The first item follows directly from the definition of $p$--variation; for a proof of items~(b,c), see~\cite{friz_rough_paths_2009}, Proposition~5.3, Proposition~5.5, respectively; items~(d,e) are easy consequences of the definition of $p$--variation.
\begin{lemma}
\label{lem:youngcontinuity} 
Let $y \in C(I, V)$ and $x_n$ in $C(I, L(V, W))$ for all $n \in \N$, with $\sup_n \var{x_n}_p$ as well as $\var{y}_q$ being finite and $\frac{1}{p} + \frac{1}{q} > 1$.
Then $x_n \to x$ uniformly implies 
\beqn{equ:youngcontinuity}
\int_0^t x_n(s) \: \dd y(s) \to \int_0^t x(s) \: \dd y(s)
\eeq
uniformly as a function of the upper limit $t$.
\end{lemma}
This is a consequence of Proposition~6.13 in~~\cite{friz_rough_paths_2009}.
\begin{lemma}
\label{lem:varwiener} 
For any $p>2$, the $p$--variation of the observations $\var{\eta}_p$ is a measurable and almost surely finite random variable.
\end{lemma}
\begin{proof}
We are assuming $\eta = \zeta + W$, and since $\zeta$ is absolutely continuous, it has finite variation of any order by Lemma~\ref{lem:varproperties}(a).
It is well known that almost any path of the Wiener process is $\frac{1}{p}$--H\"{o}lder for any $p > 2$ on compact intervals, which implies finite $p$--variation for $p>2$.
It follows directly from the definition and Minkowski's inequality that $p$--variation is subadditive, hence $\var{\eta}_p$ is almost surely finite.
\end{proof} 
In view of Lemma~\ref{lem:varwiener}, we can find a set $\Omega_0 \subset \Omega$ with $\P(\Omega_0) = 1$ so that whenever $\omega \in \Omega_0$, we have $\var{\eta}_{p} < \infty$ for any $p > 2$. 
Without loss of generality, we will henceforth assume $\Omega = \Omega_0$.
In other words, we can simply assume the observations to be a continuous function $\eta(t)$ having finite $p$-variation for any $p > 2$.
This allows us to forget about the stochastic character of the observations and instead approach data assimilation pathwise for each realisation of the observations.
In particular, we obtain from Lemma~\ref{lem:youngproperties} that for any given observation path, the integral $\int_I y(t) \: \dd \eta(t)$ is well defined and finite for all functions $y$ which have finite $p$-variation with $p < 2$.
\begin{proof}[Proof of Lemma~\ref{lem:welldefined}]
Item~(a) is a basic result in the theory of ODE's.
To prove item~(b) use the remark made just prior to this proof and the fact that $t \to \psi(t, x(t))$ has bounded $\frac{1}{\kappa}$--variation with $\kappa < \frac{1}{2}$ (as we will show in the proof of Theorem~\ref{thm:existence}).
Item~(c) follows from Hypothesis~\ref{hyp:hypothesisI}\eqref{hyp:thexconvexity}, the fact that $x$ is continuous and hence bounded on $I$, and the fact that a continuous function of measurable functions is measurable.
Item~(d) follows because if $(x, u)$ is feasible with $\norm{u}_r = \infty$, the lower bound on $\phi$ in Hypothesis~\ref{hyp:hypothesisI}\eqref{hyp:thexconvexity} yields that the deterministic costs are infinite.
\end{proof}
The following Lemma will be essential in the proof of Theorem~\ref{thm:existence}, ensuring the coercivity of the stochastic costs:
\begin{lemma}
\label{lem:detailedestimate}
Suppose that there are nonegative constants $r, q, \gamma_1, \gamma_2, C_1, C_2$ with $r \geq 1$ and $1 \leq q < 2$ so that for any feasible process $(x, u)$ we have
\begin{align}
\label{equ:detailedestcon1}
\norm{\psi(., x(.))}_{\infty}
 & \leq  C_1 (1 + \norm{u}_r^{\gamma_1}) \\
\label{equ:detailedestcon2}
\var{\psi(., x(.))}_{q}
 & \leq  C_2 (1 + \norm{u}_r^{\gamma_2}).
\end{align}
If we let
\beqn{equ:theta1theta2}
\begin{split}
\theta_1 & := \min\{\gamma_1 (1 - \frac{q}{2}) + \gamma_2 \frac{q}{2}, \gamma_2\}\\
\theta_2 & := \max\{\gamma_1 (1 - \frac{q}{2}) + \gamma_2 \frac{q}{2}, \gamma_2\}
\end{split}
\eeq
then for any $\theta$ with $\theta_1 < \theta < \theta_2$
there is a random constant $D_{\theta}$ so that
\[
\abs{\int_I \psi(t, x(t)) \: \dd \eta_t }
\leq D_{\theta} \left( 1 
    + \norm{u}_r^{\theta} \right).
\]
\end{lemma}
\begin{proof}
Fix a $\theta$ as stated and set $\tau$ so that $\theta = \gamma_1 (1 - \tau) + \gamma_2 \tau $.
Then $\frac{q}{2} < \tau < 1$, and if we define $p = \frac{q}{\tau}$, we obtain $q < p < 2$, so we can assume that $p = \frac{2}{1 + 2\epsilon}$ for some $\epsilon > 0$.
Put $\hat{p} = \frac{2}{1 - \epsilon} > 2$ and observe that $\frac{1}{\hat{p}} + \frac{1}{p} = \frac{1 - \epsilon}{2} + \frac{1 + 2\epsilon}{2} = 1 + \epsilon > 1$.
In view of Lemma~\ref{lem:varwiener}, we can invoke Lemma~\ref{lem:youngproperties} for $\var{\psi(., x(.))}_{p}$ and $\var{\eta}_{\hat{p}}$ which gives that
\[
\abs{\int_I \psi(t, x(t)) \: \dd \eta_t }
\leq C_p\left( 1 + \var{\psi(., x(.))}_{p}\right)
\]
for some random constant $C_p$ which also depends on the choice of $p$.
Next, the interpolation inequality (Lemma~\ref{lem:varproperties}.c) yields
\[
\abs{\int_I \psi(t, x(t)) \: \dd \eta_t }
\leq C_p\left( 
    1 + \var{\psi(., x(.))}_{q}^{\frac{q}{p}}
    \norm{\psi(., x(.))}_{\infty}^{1 - \frac{q}{p}} 
\right).
\]
Now use that $p = \frac{q}{\tau}$ as well as the estimates~(\ref{equ:detailedestcon1},\ref{equ:detailedestcon2}) to complete the proof.
\end{proof}
Note that bounds on the stochastic costs could also be obtained through integration by parts, as already mentioned in the introduction. 
This requires $\psi$ to have a derivative with respect to $x$ to begin with, but even if this derivative is bounded, then under the assumed conditions the stochastic term is bounded by the nonlinearity estimate~\eqref{hyp:thexnonlin}, which can be much worse than the estimates given by Lemma~\ref{lem:detailedestimate}.
We are now ready to prove Theorem~\ref{thm:existence}.
\begin{proof}[Proof of Theorem~\ref{thm:existence}]
Let $(x_n, u_n), n = 0, 1, \ldots$ be a minimising sequence with $\norm{u_n}_r < \infty$ for all $n$ and $(x_0, u_0)$ being the admissible process which exists by assumption.
We know that the costs $A(x_n, u_n)$ are bounded above (by $A(x_0, u_0) < \infty$).
Hypothesis~\ref{hyp:hypothesisI}.(\ref{hyp:thexconvexity}) implies that for any process 
\beq{equ:lowerboundA}
A(x, u) \geq C_{\phi} \int_I 
    \left(\abs{u(t)}^r - \abs{x(t)}^{\delta}\right)
    \: \dd t 
+ \int_I \psi(t, x(t)) \: \dd \eta_t,
\eeq
and using the energy estimate~\eqref{hyp:thexenergy} and taking into account that $\gamma \delta < r$ (Hyp.~\ref{hyp:hypothesisI},(\ref{hyp:thexvarbounds})), we obtain for the first integral
\beq{equ:lowerboundAref}
\int_I 
    \left(\abs{u(t)}^r - \abs{x(t)}^{\delta}\right) \: \dd t
\geq C \norm{u(t)}^r - b
\eeq
for some $C > 0$ and $b \in \R$.
We aim to control the second integral in Equation~\eqref{equ:lowerboundA} using Lemma~\ref{lem:detailedestimate} with $q = 1/\kappa$.
By Hypothesis~\ref{hyp:hypothesisI},(\ref{hyp:thexhoel}) and the chain rule in Lemma~\ref{lem:varproperties} we obtain
\[
\var{\psi(., x(.))}_q \leq K_{\norm{x}_{\infty}} C' (1 + \var{x}_1^{\kappa})
\]
for some constant $C'$.
Invoking Lemma~\ref{lem:varproperties},(a) as well as the estimates~(\ref{hyp:thexhder},\ref{hyp:thexnonlin},\ref{hyp:thexenergy}) we obtain
\beq{equ:varhmaintheorem}
\var{\psi(., x(.))}_q \leq C'' (1 + \norm{u}_r^{\beta \kappa + \gamma \alpha}).
\eeq
Again by Hypothesis~\ref{hyp:hypothesisI},(\ref{hyp:thexhoel}) and the estimate~(\ref{hyp:thexhder}) we get
\[
\norm{\psi(., x(.))}_{\infty} \leq C''' (1 + \norm{x}_{\infty}^{\kappa + \alpha}),
\]
and invoking the energy estimate~(\ref{hyp:thexenergy}) once again this becomes
\beq{equ:normhmaintheorem}
\norm{\psi(., x(.))}_{\infty} \leq C'''' (1 + \norm{u}_r^{\gamma(\kappa + \alpha)}).
\eeq
In equations~(\ref{equ:varhmaintheorem},\ref{equ:normhmaintheorem}), we can use Lemma~\ref{lem:detailedestimate} with $q := 1/\kappa$, $\gamma_1 := \gamma(\kappa + \alpha)$ and $\gamma_2 := \beta \kappa + \gamma \alpha$.
We obtain $\theta_1 = \gamma(\alpha + \kappa - \frac{1}{2}) + \frac{\beta}{2}$ and $\theta_2 = \beta \kappa + \gamma \alpha$ (note that indeed $\theta_1 \leq \theta_2$ because $\gamma \leq \beta$ as discussed just after Hypothesis~\ref{hyp:hypothesisI}).
Now according to the lower bound on $r$ in Hypothesis~\ref{hyp:hypothesisI},(\ref{hyp:thexvarbounds}), we can find $\theta$ so that $\theta_1 < \theta = r-\epsilon$ for some $\epsilon > 0$ and therefore by Lemma~\ref{lem:detailedestimate} there exists a constant $D$ only depending on $r-\epsilon$ and the observations so that 
\[
\abs{\int_I \psi(t, x(t)) \: \dd \eta_t }
\leq D \left( 1 
    + \norm{u}_r^{r-\epsilon} \right).
\]
Using this as well as the estimate~\eqref{equ:lowerboundAref} in Equation~\eqref{equ:lowerboundA}, we are finally able to find a constant $C_A$ so that 
\beq{equ:fulllowerboundA}
A(x, u) \geq C_A (1 + \norm{u}_r^r ).
\eeq
Since $A(x_n, u_n)$ is bounded, we obtain from Equation~\eqref{equ:fulllowerboundA} that $\norm{u_n}_r$ must be bounded.
Hypothesis~\ref{hyp:hypothesisI}.(\ref{hyp:thexvarbounds}) then implies that $\norm{x_n}_{\infty}$ and $\norm{\dot{x}_n}_r$ are bounded.
Since $L_r$ is reflexive, by weak compactness, we can assume (taking subsequences which we do not relabel) that $u_n \to u_*$ and $\dot{x}_n \to v_*$.
The boundedness of $\norm{\dot{x}_n}_r$ implies that $x_n$ is equicontinuous, so applying Arzela--Ascoli (and again taking subsequences), we can assume $x_n \to x_*$ uniformly as well.
(Note that a subsequence of a minimising sequence is still a minimising sequence.)
Further 
\[
x_*(t) = \xi + \int_0^t v_*(s) \: \dd s
\]
so $v_* = \dot{x}_*$.
Since the functions $x_n$ are uniformly bounded, we obtain from Hypothesis~\ref{hyp:hypothesisI}\eqref{hyp:thexconvexity} that the functions $t \to \phi(t, x_n(t), u_n(t))$ admit a uniform lower bound $\phi_0$.
We can therefore replace $\phi$ with $\max \{\phi, \phi_0 \}$ on $I \times E \times U$ without affecting continuity or convexity.
We can now apply the Integral semicontinuity theorem of~\cite{clarke_functional_analysis_2013}, (Th.~6.38) to conclude that 
\beq{equ:lsccosts}
\int_I \phi(t, x_*(t), u_*(t)) \: \dd t 
\leq \lim \inf \int_I \phi(t, x_n(t), u_n(t)) \: \dd t
\eeq
The continuity of $\psi$ together with the uniform convergence of $x_n$ implies that $\psi(., x_n(.)) \to \psi(., x_*(.))$ uniformly on $I$. 
The boundedness of $\var{\psi(., x_n(.))}_q$ follows from the boundedness of $\norm{u_n}_r$ and the estimate~\eqref{equ:varhmaintheorem}.
It then follows from Lemma~\ref{lem:youngcontinuity} that  
\beqn{equ:lscyoung}
\int_I \psi(t, x_n(t)) \: \dd \eta(t) 
\to \int_I \psi(t, x_*(t)) \: \dd \eta(t) 
\eeq
This fact together with~(\ref{equ:lsccosts}) implies
\beqn{equ:lscallcosts}
\begin{split}
& \int_I \phi(t, x_*(t), u_*(t)) \: \dd t
 - \int_I \psi(t, x_*(t)) \: \dd \eta(t) \\ 
& \leq \lim \inf \int_I \phi(t, x_n(t), u_n(t)) \: \dd t
 - \lim \int_I \psi(t, x_n(t)) \: \dd \eta(t) \\
& \leq \lim \inf \left\{ \int_I \phi(t, x_n(t), u_n(t)) \: \dd t
 - \int_I \psi(t, x_n(t)) \: \dd \eta(t)  \right\}\\
& = \inf A(x, u)
\end{split}
\eeq
where the last $\inf$ is over all admissible trajectories with $\norm{u}_r < \infty$ and the last equality follows because $(x_n, u_n)$ is a minimising sequence.
It remains to show that $(x_*, u_*)$ is a feasible process. 
The proof is the same as in~\cite{clarke_functional_analysis_2013}, Theorem~23.11. 
%
%
\end{proof}
%
%
%
\section{A maximum principle}
\label{sec:maximumprinciple}
The aim of this section is to prove that under suitable conditions, global minimisers of Problem~VAR satisfy a set of necessary conditions akin to the celebrated Pontryagin maximum principle in optimal control, see for instance~\cite{fleming_controlled_2006},~Sec.I.6.
The classical Pontryagin maximum principle requires certain regularity of the problem, in particular $f$ and $g$ need to be more regular than we have assumed so far.
Recently, there has been considerable progress in relaxing these conditions, at the price of having to use more general notions of derivatives as well as some heavy nonsmooth analysis machinery, see~\cite{bardi_optimal_control_2009},~Sec.III.3.4.\ and in particular~\cite{clarke_functional_analysis_2013},~Th.22.26.
Taking one step at a time, we will prove the maximum principle under essentially classical regularity assumptions, but admitting systems bearing the essential features of the example in Section~\ref{sec:motivatingexample}.
Our main problem will again be dealing with the stochastic part of the costs in an appropriate manner.
We note that here again, this can be attempted by integrating the stochastic costs by parts, which is now permitted in view of Hypothesis~\ref{hyp:hypothesisII}\eqref{hyp:thmaxhoel}.
This results in a classical optimal control problem, but applying the standard maximum principle then requires that the running costs of this problem have one further derivative with respect to $x$, which is not guaranteed by our assumptions.
The following proof however shows that the stochastic costs can be dealt with directly.
\begin{proof}[Proof of Theorem~\ref{thm:maximumprinciple}]
We will only sketch the proof, focussing on the bits relevant for the stochastic costs which are nonstandard. 
A few technical details will be left to subsequent Lemmas.
Our argument closely follows the proof of Theorem~2.1 in Chapter~2 of~\cite{yong_stochastic_controls_1999}.
The proof of Theorem~6.3 in Chapter~I of~\cite{fleming_controlled_2006} uses very similar ideas.\footnote{The proof however seems to contain an error with regards to the variational equation.}
As in these references, we use a needle variation.
Let $s \in I\setminus \{T\}$ be a Lebesgue point for $u$.
(Recall that $s$ is a Lebesgue point for $u$ if 
\mbox{$
\lim_{\epsilon \to 0} \frac{1}{2\epsilon} 
\int_{[-\epsilon, \epsilon]}
    u(s + t)
\; \dd t
 = u(s),
$}
and that if $u$ is integrable over $I$, almost every $s \in I$ is a Lebesgue point of $u$.) 
For some $v \in U$ and $\epsilon > 0$ so that $s + \epsilon \in I$, define the interval $I_{\epsilon} = [s, s + \epsilon]$ and the control
\beqn{equ:needlevariation}
u_{\epsilon}(t) = \begin{cases}
v & \text{ if $t \in I_{\epsilon}$}\\
u(t) & \text{ else.}
\end{cases}
\eeq
Note that $u_{\epsilon} \in \mathcal{U}_r$.
Let $x_{\epsilon}$ be the (unique) solution to the state equation~\eqref{equ:initvalproblem} with $u = u_{\epsilon}$ and $x_{\epsilon}(0) = \zeta$.
To be able to write the state equation in a more compact form in the following, we introduce the function $h(t, x, u) := f(t, x) + g(t, x) u$.
Using this abbreviation we can write 
\beqn{equ:3.10}
\dot{x}_{\epsilon} (t)
 = h(t, x_{\epsilon}(t), u(t))
 + g(t, x_{\epsilon}(t)) (v - u(t)) \cf_{I_{\epsilon}}(t).
\eeq
The main idea of the proof is to consider a Taylor expansion of the state equation and the costs to first order in the perturbation $t \to g(t, x(t)) (v - u(t)) \cf_{I_{\epsilon}}(t)$.
To this end, we define
\beqn{equ:3.20}
\dot{\zeta}_{\epsilon} (t)
 = \DD_2 h(t, x(t), u(t)) \zeta_{\epsilon}(t)
 + g(t, x(t)) (v - u(t)) \cf_{I_{\epsilon}}(t), 
 \quad \zeta_{\epsilon}(0) = 0.
\eeq
We make the following claims:
\beq{equ:mxproof_1}
\norm{x_{\epsilon} - x}_{\infty} = O(\epsilon),
\eeq
%
%
\beq{equ:mxproof_2}
\norm{x_{\epsilon} - x - \zeta_{\epsilon}}_{\infty} = o(\epsilon),
\eeq
and
\beq{equ:mxproof_3}
\begin{split}
A(x_{\epsilon}, u_{\epsilon}) - A(x, u)
& = \int_I \DD_2 \phi(t, x(t), u(t)) \zeta_{\epsilon}(t) \: \dd t
    + \int_I \DD_2 \psi(t, x(t)) \zeta_{\epsilon}(t) \: \dd \eta_t \\
& \quad + \int_{I_{\epsilon}} \phi(t, x(t), v) - \phi(t, x(t), u(t)) \: \dd t
    + o(\epsilon).
\end{split}
\eeq
The proof of these claims is virtually identical to the proof of Lemma~2.2 in Chapter~2 of~\cite{yong_stochastic_controls_1999}, with the exception of the stochastic cost contribution to Equation~\eqref{equ:mxproof_3}, which will be considered in Lemma~\ref{lem:stochcostvariation}.
As will be shown in Lemma~\ref{lem:costate}, the costate equation~\eqref{equ:costate} has a unique solution $\lambda$ which has finite $q$--variation for any $q > 2$.
Very much like in the classical maximum principle, the functions $\zeta_{\epsilon}$ and $\lambda$ are in duality, leading to the identity
\beqn{equ:3.30}
\begin{split}
& \int_{I} \lambda(t) g(t, x(t)) (v - u(t)) \cf_{I_{\epsilon}}(t) \: \dd t\\
& = \int_I \DD_2 \phi(t, x(t), u(t)) \zeta_{\epsilon}(t) \: \dd t
+ \int_I \DD_2 \psi(t, x(t)) \zeta_{\epsilon}(t) \: \dd \eta_t. \end{split}
\eeq
This will be proved in Lemma~\ref{lem:duality}.
We use this relation as well as the definition of the Hamiltonian $H$ in Equation~\eqref{equ:mxproof_3} and invoke the optimality of $u$ to obtain
\beqn{equ:3.40}
0 \leq \int_{I} \cf_{I_{\epsilon}}(t) \left\{ H(t, x(t), \lambda(t), v) -  H(t, x(t), \lambda(t), u(t)) \right\} \: \dd t
+ o(\epsilon).
\eeq
The result now follows if we divide by $\epsilon$ and send $\epsilon \to 0$, keeping in mind that $s$ is a Lebesgue point of $u$
\end{proof}
\begin{lemma}
\label{lem:stochcostvariation}
With the definitions as in the proof of Theorem~\ref{thm:maximumprinciple}, we have
\beqn{equ:3.60}
\int_I \psi(t, x_{\epsilon}(t)) \: \dd \eta_t
 - \int_I \psi(t, x(t)) \: \dd \eta_t
 - \int_I \DD_2 \psi(t, x(t)) \zeta_{\epsilon}(t) \: \dd \eta_t
 = o(\epsilon).
\eeq
\end{lemma}
\begin{proof}
According to Lemma~\ref{lem:youngcontinuity}, this follows if 
\[
\frac{1}{\epsilon} \big[
    \psi(t, x_{\epsilon}(t))
     - \psi(t, x(t))
     - \DD_2 \psi(t, x(t)) \cdot \zeta_{\epsilon}(t)\big]
\longrightarrow 0
\]
uniformly in $t$, and for some $q < 2$ the $q$--variation of the left hand side stays bounded.
By Taylor's theorem, we can write the left hand side as
\beq{equ:stochcostvariation1}
\begin{split}
& \int_0^1 
    \DD_2 \psi \big(t, \tau x_{\epsilon}(t) 
        + (1-\tau) x(t) \big)
    - \DD_2 \psi \big(t, x(t) \big)
    \: \dd \tau 
    \cdot \frac{1}{\epsilon} \big[x_{\epsilon}(t) - x(t)\big]\\
& + \DD_2 \psi \big(t, x(t) \big) \frac{1}{\epsilon} \big[x_{\epsilon}(t) - x(t) - \zeta_{\epsilon}(t)\big].
\end{split}
\eeq
Now using Equation~\eqref{equ:mxproof_1} and the fact that $\DD_2 \psi$ is uniformly continuous (by Hypothesis~\ref{hyp:hypothesisII}\eqref{hyp:thmaxhoel}), we obtain that the first term goes to zero uniformly. 
Equation~\eqref{equ:mxproof_2} guarantees that the second term goes to zero uniformly as well. 
The proof of Equation~\eqref{equ:mxproof_1} reveals, upon closer inspection, that the $L_r$--norm of $\frac{1}{\epsilon} \frac{\dd}{\dd t}[x_{\epsilon} - x]$ remains bounded, and the same is true for $\frac{1}{\epsilon} \frac{\dd}{\dd t} \zeta_{\epsilon}$.
By Lemma~\ref{lem:varproperties}(a) this implies that the $q$--variation of these functions remains bounded for any $q \geq 1$.
By Lemma~\ref{lem:varproperties}(d) this leaves to show that also the term
\beqn{equ:stochcostvariation_part}
\int_0^1 
    \DD_2 \psi \big(t, \tau x_{\epsilon}(t) 
        + (1-\tau) x(t) \big)
    \: \dd \tau 
\eeq
in expression~\eqref{equ:stochcostvariation1} has $q$--variation bounded in $\epsilon$.
Define the function $y_{\epsilon}(\tau, t) = \DD_2 \psi \big(t, \tau x_{\epsilon}(t) + (1-\tau) x(t) \big)$.
By Hypothesis~\ref{hyp:hypothesisII},(\ref{hyp:thmaxhoel}), the $q$--variation of $t \to y_{\epsilon}(\tau, t)$ is bounded uniformly in $\epsilon$ and $\tau$, for $q < \frac{1}{2}$.
But it is easily seen that if $\var{y_{\epsilon}(\tau, .)}_q \leq C$, then $t \to \int_0^1 y_{\epsilon}(\tau, t) \: \dd \tau$ has $q$--variation bounded by $C$ as well.
We conclude that as a function of $t$, the expression in Equation~\eqref{equ:stochcostvariation1} has $q$--variation bounded in $\epsilon$.
\end{proof}
\begin{lemma}
\label{lem:costate}
There is a unique solution to the costate equation~\eqref{equ:costate} which has finite $q$--variation for $q > 2$.
\end{lemma}
\begin{proof}
According to Hypothesis~\ref{hyp:hypothesisII}\eqref{hyp:thmaxinteg} the function $t \to \DD_2 \phi(t, x(t), u(t))$ is integrable.
Further, from Hypothesis~\ref{hyp:hypothesisII}\eqref{hyp:thmaxhoel}, along with the absolute continuity of $x$ and the chain rule in Lemma~\ref{lem:varproperties}, we can conclude that $t \to \DD_2 \psi(t, x(t))$ has $q$--variation for some $q < 2$.
These two facts imply that the second and third integral in the definition of $\lambda$ (Equ.~\ref{equ:costate}) are well defined and continuous as a function of the lower limit.
We can thus write Equation~\eqref{equ:costate} in the form
\[
\lambda(t) = \int_t^T \lambda(s) M(s)  \: \dd s + F(t)
\]
where $M(t) := \DD_2 f(t, x(t)) + \DD_2 g(t, x(t)) u(t)$ and $F$ is continuous with finite $q$--variation for $q > 2$. 
Noting that $\abs{M}$ is integrable, it follows by standard ODE arguments that $\lambda$ is well defined and unique.
Further, $\lambda$ is the sum of an absolutely continuous part and a part with finite $q$--variation for $q > 2$.
From this, it is easy to see that $\lambda$ itself has finite $q$--variation for $q > 2$.
\end{proof}
\begin{lemma}[Duality Lemma]
\label{lem:duality}
Let $I = [0, T]$ be an interval and $V$ a finite dimensional vector space with norm $\abs{.}$ and dual space $V'$. 
Consider functions $a: I \to V$ and $b: I \to V'$ so that both $\var{a}_p$ and $\var{b}_q$ are finite, where $\frac{1}{p} + \frac{1}{q} > 1$.
Further, consider a measurable function $M: I \to L(V, V)$ so that the function $t \to \abs{M(t)}$ is integrable (here $\abs{.}$ denotes the operator norm).
Then the equations
\begin{align*}
\zeta(t) & = \zeta(0) + \int_0^t M(s) \zeta(s) \: \dd s + a(t), \\
\lambda(t) & = \lambda(T) + \int_t^T \lambda(s) M(s) \: \dd s + b(t)
\end{align*}
have unique solutions, and both $\var{\zeta}_p$ and $\var{\lambda}_q$ are finite.
Further, the duality relation
\beqn{eqn:3.80}
\lambda(T) \zeta(T) - \lambda(0) \zeta(0) = \int_I \zeta(t) \: \dd b(t) + \int_I \lambda(t) \: \dd a(t)
\eeq
holds.
\end{lemma}
\begin{proof}
The uniqueness of $\zeta$ and $\lambda$ as well as the claims about $\var{\zeta}_p$ and $\var{\lambda}_q$ follow as in the proof of Lemma~\ref{lem:costate}.
A direct calculation gives
\beqn{equ:3.100}
\int_I \lambda(t) \: \dd \zeta(t)
= \int_I \lambda(t) M(t) \zeta(t) \: \dd t + \int_I \lambda(t) \: \dd a(t)
\eeq
and similarly
\beqn{equ:3.110}
\int_I \zeta(t) \: \dd \lambda(t)
= -\int_I \lambda(t) M(t) \zeta(t) \: \dd t + \int_I \zeta(t) \: \dd b(t).
\eeq
Adding these two relations gives
\beqn{equ:3.120}
\int_I \zeta(t) \: \dd \lambda(t) + \int_I \lambda(t) \: \dd \zeta(t)
= \int_I \zeta(t) \: \dd b(t) + \int_I \lambda(t) \: \dd a(t),
\eeq
and integrating by parts on the left hand side gives the duality relation.
\end{proof}
\section{Regularity of controls}
\label{sec:regularitycontrols}
To investigate the regularity of controls, we assume that for every $(t, z, \mu) \in I \times E \times E'$ the function $H(t, z, \mu, .)$ has a unique minimiser, that is, there is a function $\upsilon: I \times E \times E' \to U$ so that
\[
H(t, z, \mu, \upsilon(t, z, \mu)) = \inf_{u \in U} H(t, z, \mu, u). 
\]
By the characterisation~\eqref{equ:maxcondition} of optimal controls, regularity properties of $\upsilon$ can be translated into regularity properties of optimal controls. 
Since the regularity of $\upsilon$ is determined by the Hamiltonian $H$ which neither contains the stochastic running costs nor the observations, we can rely on classical results, the proof of which we will only sketch.
\begin{definition}
The deterministic running costs are {\em strongly convex} in $u$ if for every bounded subset $C \subset I \times E \times U$ there is a $c > 0$ so that whenever $(t, z, w_1)$ and $(t, z, w_2)$ are in $C$, we have
\beqn{equ:4.10}
\left( \DD_3 \phi(t, z, w_1) - \DD_3 \phi(t, z, w_2)\right)
    (w_1 - w_2)
\geq c \abs{w_1 - w_2}^2.
\eeq
Note that $H$ is strongly convex if $\psi$ is convex, meaning in particular that $\upsilon$ is well defined in this case (see also~\cite{clarke_functional_analysis_2013}).
\end{definition}
\begin{theorem}[Regularity of controls]
\label{thm:controlregularity}
Suppose Hypothesis~\ref{hyp:hypothesisII} is in force.
\begin{enumerate}
\item If $\upsilon$ is well defined, then an optimal control has a continuous modification.
\item If $\psi$ is strongly convex, and $\DD_3 \phi$ as well as $g$ are locally Lipschitz in $(t, x)$, then an optimal control has a continuous modification that has bounded $q$--variation for any $q > 2$.
\end{enumerate}
\end{theorem}
\begin{proof}
If $(x, u, \lambda)$ is an optimal triple for initial condition $\xi$, then by the definition of $\upsilon$ we must have
\[
u(t) = \upsilon(t, x(t), \lambda(t))
\]
for almost all $t \in I$.
Hence item~1 follows if we can show that $\upsilon$ is continuous, while establishing that $\upsilon$ is locally Lipschitz, together with Lemmas~\ref{lem:costate} and~\ref{lem:varproperties}(e) will prove item~2.
Our proof of these two properties of $\upsilon$ follows~\cite{clarke_functional_analysis_2013}, Theorem~23.17.
With $w$ be an arbitrary element of $U$, we have the estimate
\[
H(t, z, \mu, w) 
\geq H(t, z, \mu, \upsilon(t, z, \mu)) 
\geq C_{\phi} \abs{\upsilon(t, z, \mu)}^r - \abs{z}^{\delta} + \mu \cdot \left[f(t, z) + g(t, z) \upsilon(t, z, \mu)\right],
\]
and because $f, g$ are continuous, we can infer that $\upsilon$ is bounded on bounded sets.
Consider a sequence $\{(t_n, z_n, \mu_n), n \in \N\}$ so that $(t_n, z_n, \mu_n) \to (t, z, \mu)$.
Taking any subsequence, there is a subsubsequence so that $\upsilon(t_n, z_n, \mu_n) \to w$ because this sequence is bounded.
Then by definition of $\upsilon$
\[
H(t_n, z_n, \mu_n, \upsilon(t_n, z_n, \mu_n))
\leq  H(t_n, z_n, \mu_n, \upsilon(t, z, \mu))
\]
and taking limits we obtain
\[
H(t, z, \mu, w)
\leq  H(t, z, \mu, \upsilon(t, z, \mu))
\]
whence $w = \upsilon(t, z, \mu)$, and we can conclude that $\upsilon$ is continuous.
To prove Theorem~\ref{thm:controlregularity}(b), consider $(t, z_1, \mu_1)$ and $(s, z_2, \mu_2)$ in $I \times E \times E'$ and use the shorthands $w_1 := \upsilon(t, z_1, \mu_1)$ and $w_2 := \upsilon(s, z_2, \mu_2)$.
Note that due to the convexity of $H$ in $u$ and the defintion of the function $\upsilon$ we must have
\[
\DD_4 H(t, z_1, \mu_1, w_1) (w_2 - w_1) \geq 0 
\quad \text{and} \quad
\DD_4 H(s, z_2, \mu_2, w_2) (w_1 - w_2) \geq 0. 
\]
The strong convexity condition implies
\beqn{equ:4.20}
\begin{split}
c \abs{w_1 - w_2}^2 
& \leq \left( \DD_4 H(t, z_1, \mu_1, w_1) 
    - \DD_4 H(t, z_1, \mu_1, w_2)\right)(w_1 - w_2)\\
& = \left( \DD_4 H(t, z_1, \mu_1, w_1) 
    - \DD_4 H(s, z_2, \mu_2, w_2) \right.\\
& \quad \left. + \DD_4 H(s, z_2, \mu_2, w_2)
    - \DD_4 H(t, z_1, \mu_1, w_2)\right)(w_1 - w_2)\\
& \leq \left( \DD_4 H(s, z_2, \mu_2, w_2)
    - \DD_4 H(t, z_1, \mu_1, w_2)\right)(w_1 - w_2)\\
& \leq L \left( \abs{s - t} + \abs{z_2 - z_1} + \abs{\mu_2 - \mu_1}\right) \abs{w_1 - w_2}.
\end{split}
\eeq
Hence, $\upsilon$ is Lipschitz, and by Lemma~\ref{lem:varproperties}(e), together with the fact that $x$ is absolutely continuous and that $\lambda$ has $q$--variation for $q > 2$, we can conclude that the control also has $q$--variation for $q > 2$.
\end{proof}
\section{The value function}
\label{sec:uniqueness}
In classical optimal control, there are well known relationships between the (generalised) derivative of the value function and the costate, permitting one to study the uniqueness of optimal controls, among other things~(see for instance~\cite{frankowska_singularities_value_function_2005}).
In this section, we will present a few results in this direction. 
Crucial to this analysis is the fact that if $(x, u, \lambda)$ is an optimal triple for some $\xi \in E$, then $(x, \lambda)$ satisfy a Hamiltonian ODE (Eq.~\eqref{equ:hamiltoniansystem} in Thm.~\ref{thm:hamiltoniansystem} below), which is obtained by using the minimum condition~\eqref{equ:maxcondition} to eliminate the control $u$ from the state and costate equation.
Further, it is required that solutions to Equation~\eqref{equ:hamiltoniansystem} are unique with respect to the initial condition $(x(0), \lambda(0))$.
Currently, we are only able to prove this under the assumption that the function $(t, z) \to \DD_2 \psi(t, z)$ has continuous partial derivatives with respect to both arguments which are Lipschitz in $z$.
We conjecture though that weaker conditions might be sufficient; this will be investigated in a future paper. 
Given our incomplete understanding of these equations at this point, we have to impose these properties as assumptions, and pending a better understanding, the results in this section have to be regarded as preliminary. 
Throughout this section, we fix a $\xi_0 \in E$ and assume there is an admissible control for $\xi_0$.
Further we impose
\begin{hypothesis}
\label{hyp:hypothesisIII}
Hypothesis~\ref{hyp:hypothesisII} and 
\begin{enumerate}
\item The mapping $\upsilon$ defined in Section~\ref{sec:regularitycontrols} is well defined.
\item There is a neighbourhood $X \subset E$ of $\xi_0$ so that Hypothesis~\ref{hyp:hypothesisI}(e) is valid for all $\xi \in X$.
\end{enumerate}
\end{hypothesis}
We first show that the necessary conditions in
Theorem~\ref{thm:maximumprinciple} can be recast as a Hamiltonian
system with stochastic perturbation.
Consider the function
\[
m(t, z, \mu) 
:= \inf_{u \in U} H(t, z, \mu, u) 
= H(t, z, \mu, \upsilon(t, z, \mu))
\]
which is continuous.
From the fact that $\upsilon$ is well defined (i.e.\ $H$ has a unique minimiser in $u$) and is bounded on bounded sets, and because $H$ has continuous derivatives with respect to $z, \mu$, we can conclude that also $m$ has continuous derivatives with respect to $z, \mu$ given by
\[
\DD m(t, z, \mu) 
= \DD H(t, z, \mu, u) |_{u = \upsilon(t, z, \mu)}
\]
where $\DD$ is the partial derivative with respect to $z$ or $\mu$.
The proof of this well known fact is omitted.
We immediately obtain the following theorem:
\begin{theorem}[Hamiltonian system]
\label{thm:hamiltoniansystem}
Let $(x, u, \lambda)$ be an optimal triple for initial condition $\xi$.
Then the state and costate $x, \lambda$ satisfy the equation
\beq{equ:hamiltoniansystem}
\begin{split}
\dot{x}(t) & = \DD_3 m(t, x(t), \lambda(t))\\
\lambda(t) & = \int_t^T \DD_2 m(s, x(s), \lambda(s)) \: \dd s 
    + \int_t^T \DD_2 \psi(s, x(s)) \: \dd \eta_s\\
x(0) & = \xi.
\end{split}
\eeq
\end{theorem}
To define the value function, one usually extends the Problem~VAR to a class of problems by asking for optimal controls on the interval $[s, T]$ with $s \geq 0$ and initial condition $x(s) = \xi$, but we will not do this here.
To any $u \in \mathcal{U}_r$ and any $\xi \in X$ there corresponds a unique absolutely continuous function $x:I \to E$ so that $(x, u)$ is a feasible process with respect to $\xi$.
The costs of this process can be regarded as a function $J(\xi, u)$ on $X \times \mathcal{U}_r$.
Note that $J$ is well defined but might be infinite.
\begin{definition}
\label{def:valuefunc}
On $X$ we define the {\em value function} as 
\beqn{equ:valuefunc}
V(\xi) = \inf_{u \in \mathcal{U}_r} J(\xi, u),
\eeq
with $V(\xi) = \infty $ if there is no admissible control for $x$.
\end{definition}
Regarding $J$ and the value function, we have the following 
\begin{lemma}
\begin{enumerate}
\item
There exists a neighbourhood $X_0 \subset X$ so that $J(., u)$ is bounded on $X_0$, where $u$ is the admissible control for $\xi_0$.
\item 
If $\xi \in X_0$ and $u \in \mathcal{U}_r$ is admissible for $\xi$, then $J(., u)$ has a partial derivative at $\xi$, and 
\beqn{equ:meaningofp}
\DD_1 J(\xi, u) = \lambda(0)
\eeq
where $\lambda$ is given by Equation~\eqref{equ:costate}. 
\item The value function is Lipschitz continuous on bounded subsets of $X_0$.
\end{enumerate}
\end{lemma}
\begin{proof}
Let $(x, u)$ be the admissible process for $\xi_0$.
Then due to Hypothesis~\ref{hyp:hypothesisIII}(b), there is a neighbourhood $X_0 \subset X$ of $\xi_0$ so that any solution $(y, u)$ of the state equation with control $u$ and initial condition $\xi \in X_0$ satisfies the bound $\norm{y - x}_{\infty} \leq R$, where $R$ is as in Hypothesis~\ref{hyp:hypothesisII}(c).
Further, 
\beq{equ:5.10}
\begin{split}
J(\xi, u) & = J(\xi_0, u) \\
& \quad + \int_I \phi(t, y(t), u(t)) - \phi(t, x(t), u(t))\: \dd t \\
& \quad + \int_I \psi(t, y(t)) - \psi(t, x(t))\: \dd \eta_t.
\end{split}
\eeq
From the mean value theorem, we get
\[
|\phi(t, y(t), u(t)) - \phi(t, x(t), u(t))|
\leq |\DD_2\phi(t, \theta(t), u(t))| R
\]
for some function $\theta:I \to E$ with $\norm{\theta - x}_{\infty} \leq R$.
It can be shown (see~\cite{clarke_functional_analysis_2013}, proof of theorem~22.17) that Hypothesis~\ref{hyp:hypothesisII}(c) implies the apparently stronger statement that if $\abs{z - x(t)} \leq R$, then $|\DD_2\phi(t, z, u(t))| \leq c'|\phi(t, x(t), u(t))| + d'(t)$ almost surely, with $c' \geq 0$ and $d'$ integrable.
This shows that the first integral in Equation~\eqref{equ:5.10} is bounded. 
The stochastic costs are easily seen to be continuous with respect to varying the initial condition of the state equation, since this causes the trajectories to vary continuously in the uniform topology, while the 1-variation remains bounded.
Hence, all terms on the right hand side of Equation~\eqref{equ:5.10} remain bounded. 
To prove item~(b), let $\lambda$ be the unique solution of Equation~\eqref{equ:costate}.
We regard $x$ as a function of time and the initial condition $\xi$ and make this explicit by writing $x_{\xi}(t)$ for $t \in I$.
Evidently, $x_{\xi}(0) = \xi$.
For any $v \in E$ we have 
\beq{equ:derivativeJ}
\DD_1 J(\xi, u) v 
= \int_0^T \DD_2 \phi(t, x_{\xi}(t), u(t)) \, \zeta(t) \: \dd t 
+ \int_0^T \DD_2 \psi(t, x_{\xi}(t)) \, \zeta(t) \: \dd \eta_t
\eeq
where 
\[
\frac{\dd }{\dd t}\zeta(t) = \DD_2 \{f(t, x_{\xi}(t)) 
    + g(t, x_{\xi}(t))u(t) \} \zeta(t)
\qquad \zeta(0) = v.
\]
A proof of this runs similar to the proof of the claims~(\ref{equ:mxproof_1},\ref{equ:mxproof_2},\ref{equ:mxproof_3}).
Applying the Duality Lemma~\ref{lem:duality} gives that the right hand side of Equation~\eqref{equ:derivativeJ} is equal to $\trp{\lambda(0)} v$.
To prove item~(c), first note that by item~(a), there is a control $u$ that is admissible for any $\xi \in X_0$, and $J(., u)$ is bounded over $X_0$ by $J_0$, say.
This implies that for any $\xi \in X_0$ there exist an optimal triple $(x_{\xi}, u_{\xi}, \lambda_{\xi})$, and we need to show that $\lambda_{\xi}(0)$ is bounded for $\xi \in X_0$. 
The estimate~\eqref{equ:fulllowerboundA} shows that $\norm{u_{\xi}}_r$ is bounded over $\xi \in X_0$.
Going back to Hypothesis~\ref{hyp:hypothesisIII}(b) we see that $\norm{x_{\xi}}_{\infty}$ and $\norm{\dot{x}_{\xi}}_{r}$ are also bounded.
We remember that by the maximum principle (Thm.~\ref{thm:maximumprinciple}), we have
\begin{multline}
\label{equ:costate_again}
\lambda_{\xi}(0) = \int_{I} 
    \lambda(s) \big\{
            \DD_2 f(s, x_{\xi}(s)) 
            + \DD_2 g(s, x_{\xi}(s)) u_{\xi}(s)
        \big\} \: \dd s\\
    + \int_{I} 
        \DD_2 \phi(s, x_{\xi}(s), u_{\xi}(s)) \: \dd s
    + \int_{I} 
        \DD_2 \psi(s, x_{\xi}(s)) \: \dd \eta_s.
\end{multline}
The function 
$t \to \DD_2 f(s, x_{\xi}(s)) 
            + \DD_2 g(s, x_{\xi}(s)) u_{\xi}(s)$
is bounded in $\norm{.}_1$ thanks to Hypothesis~\ref{hyp:hypothesisII}(a) and the boundedness of $\norm{u_{\xi}}_r$.
The term 
$\int_{I} 
        \DD_2 \psi(s, x_{\xi}(s)) \: \dd \eta_s
$
is bounded due to Hypothesis~\ref{hyp:hypothesisII}(d), the boundedness of $\norm{x_{\xi}}_{\infty}$ and our estimates of the Young integral.
Hypothesis~\ref{hyp:hypothesisII}(c) provides the estimate
\beqn{equ:costate_secondpart}
\begin{split}
& \int_{I} \DD_2 
    \phi(s, x_{\xi}(s), u_{\xi}(s)) \: \dd s \\
& \leq \int_{I} 
    c|\phi(s, x_{\xi}(s), u_{\xi}(s))| \: \dd s + D\\
& \leq cJ_0 + c\, |\! \int_{I} 
    \psi(s, x_{\xi}(s)) \: \dd \eta_s| + D,
\end{split}
\eeq
and the integral is bounded again due to Hypothesis~\ref{hyp:hypothesisII}(d), the boundedness of $\norm{x_{\xi}}_{\infty}$ and our estimates of the Young integral.
\end{proof}
We will from now on restrict attention to $X_0$ which we rename $X$.
Since $V$ is locally Lipschitz, $\DD V$ exists on a dense set in $X$ by Rademacher's theorem.
The {\em reachable gradient} of $V$ at $\xi \in X$, denoted by $\partial^* V(\xi)$, is the set of all cluster points of $\DD V(\xi_n)$ when $\xi_n \to \xi$.
Note that $\partial^* V(\xi)$ is nonempty for all $\xi \in X$, and it is well known (see~\cite{frankowska_singularities_value_function_2005}, prop.~2) that $V$ is differentiable at $\xi$ if and only if $\partial^* V(\xi)$ is a singleton, in which case $\{\DD V(\xi)\} = \partial^* V(\xi)$.
\begin{theorem}
%
\label{thm:uniqueness}
\begin{enumerate}
\item Let $\xi \in X$ and suppose that for any control $u$ that is optimal with respect to $\xi$, the corresponding pair $(x, \lambda)$ is the unique solution of the Hamiltonian system~\eqref{equ:hamiltoniansystem} with respect to the initial condition $(\xi, \lambda(0))$.
Then if $V$ has a derivative at $\xi$, the optimal control $u$ for $\xi$ is unique (a.s.\ wrt Lebesgue measure), and $\DD V(\xi) = \lambda(0)$.
\item Assume that in addition to the conditions in the first item, solutions to the Hamiltonian system~\eqref{equ:hamiltoniansystem} depend continuously on initial conditions in the uniform topology. 
Then if $w \in \partial^{*}V(\xi)$, there exists an optimal state control pair $(x, u)$ with $\lambda(0) = w$.
In particular, if $u$ is a unique optimal control with respect to $\xi$, then $V$ has a derivative at $\xi$.
\end{enumerate}
\end{theorem}
\begin{proof}
Let $u_1$ and $u_2$ be optimal controls, with $(x_1, u_1, \lambda_1)$ and $(x_2, u_2, \lambda_2)$ being the corresponding optimal triples.
Note that $x_1(0) = x_2(0) = \xi$, and both $(x_1, \lambda_1)$ and $(x_2, \lambda_2)$ are solutions of the Hamiltonian system~\eqref{equ:hamiltoniansystem}.
Since $V(z) \leq J(z, u_1)$ for $z \in X$, with equality if $z = \xi$, we can conclude that $\DD V(\xi) = \DD_1 J(\xi, u_1)$, in case $V$ has a derivative at $\xi$.
Since the same is true for $u_2$, we see that $\lambda_1(0) = \lambda_2(0) = \DD V(\xi)$ and therefore $(x_1, \lambda_1) = (x_2, \lambda_2)$ because the Hamiltonian system~\eqref{equ:hamiltoniansystem} has unique solutions with respect to initial conditions.
Therefore $u_1(t) = \upsilon(x_1(t), \lambda_1(t)) = \upsilon(x_2(t), \lambda_2(t)) = u_2(t)$, and item~(a) is proved.
To prove item~(b), let $w \in \partial^*V(\xi)$.
Take a sequence $\xi_k \in X, k \in \N$ with $\xi_k \to \xi$ and so that $V$ is differentiable at $\xi_k$; letting $\lambda_k(0) = \DD V(\xi_k)$, we also assume $\lambda_k(0) \to w$.
This is possible owing to the definition of $\partial^*V(\xi)$.
For each $k \in \N$, let $(x_k, \lambda_k)$ be the solution to the Hamiltonian system~\eqref{equ:hamiltoniansystem} with initial condition $(\xi_k, \lambda_k(0))$.
By assumption, $(x_k, \lambda_k) \to (x, \lambda)$ uniformly, where $(x, \lambda)$ is the solution to the Hamiltonian system~\eqref{equ:hamiltoniansystem} with initial condition $(\xi, w)$.
Putting $u_k(t) = \upsilon(x_k(t), \lambda_k(t)), t \in I$, we see that $u_k$ is the unique optimal control for $\xi_k$ for each $k \in \N$.
Now $u_k$ converges uniformly to some control $u$, and it follows as in the proof of Theorem~\ref{thm:existence} that $u$ is optimal for $\xi$. 
The remainder of item~(b) follows from the remark just before the Theorem and the fact that if $u$ is the only optimal control for $\xi$, then $\partial^*V(\xi)$ must be singleton.
\end{proof}
\section{Motivating example from geophysical fluid dynamics}
\label{sec:motivatingexample}
We consider a class of examples with a state equation of the form
\beq{equ:examplestate}
\dot{x}(t) = f_1(t, x(t)) + f_2(x(t)) + g(t, x(t)) u(t)
\eeq
on the interval $I = [0, T]$ with initial condition $x(0) = \xi \in E$.
We use $E = \R^n$ for some $n$ and use the standard scalar product and norm on $E$.
We impose the conditions
\begin{enumerate}
\item \label{hyp:exmplphys}
$f_1$ is continuous and has a bounded and continuous derivative with respect to $x$.
\item \label{hyp:exmpladvec}
$f_2$ is a bilinear form with the property $\trp{x} f_2(x) = 0$.
\item \label{hyp:exmplg}
$g$ is bounded and continuous and has a bounded and continuous derivative with respect to $x$.
\setcounter{itemfreeze}{\value{enumi}}
\end{enumerate}
In the context of geophysical fluid dynamics, $f_1$ represents viscosity, coriolis forces and other physical effects.
The bilinear term $f_2$ is a kinematic term inherited from the advection term (or ``material acceleration'') in the Navier Stokes equations.
The function $g$ need not have any physical interpretation but might be present to scale the control $u$ or enforce balance conditions.
We also assume the presence of a closed and convex control set $U \subset E$ (which might be equal to $E$); again, this set might represent balance conditions. 
The class of systems described by the conditions (a-c) above contains various conceptual weather and climate models, such as Lorenz'63, Lorenz'96, and truncation approximations to the Navier~Stokes and Barotropic~Vorticity models. 
Of course, in the original form of these models, there is no $g$, which is an additional component coming in with the data assimilation.
We will now demonstrate that the conditions (a-c) above (plus further conditions on the running costs discussed below) imply Hypotheses~\ref{hyp:hypothesisI} and~\ref{hyp:hypothesisII}.
By multiplying the state equation~\eqref{equ:examplestate} with $\trp{x}$ and integrating from $0$ to $t$ we obtain
\[
\begin{split}
\frac{1}{2} x^2(t) 
 & = \frac{1}{2} \xi^2 
    + \int_0^t \trp{x(s)} f_1(s, x(s))  \: \dd s
    + \int_0^t \trp{x(s)} g(s, x)u(s)  \: \dd s\\
 & \leq C_1
    + C_2\int_0^t \abs{x(s)}^2 \: \dd s
    + C_3\int_0^t \abs{x(s)}\abs{u(s)}  \: \dd s\\
 & \leq C'_1
    + C'_2\int_0^t \abs{x(s)}^2 \: \dd s
    + C'_3\int_0^t \abs{u(s)}^2 \: \dd s
\end{split}    
\]
where we have used first the property of $f_2$, next the properties of $f_1$ and $g$, and finally Young's inequality.
Applying Gr\"{o}nwall's inequality, we obtain the energy estimate~\eqref{hyp:thexenergy} with $\gamma = 1$ and $r = 2$.
Using the energy estimate directy in the state equation~\eqref{equ:examplestate}, we obtain the nonlinearity estimate~\eqref{hyp:thexnonlin} with $\beta = 2$ and $r = 2$.
We introduce an observation function $\obsf:E \to R^d$ and impose the condition
\begin{enumerate}
\setcounter{enumi}{\value{itemfreeze}}
\item The function $\obsf$ satisfies Hypothesis~\ref{hyp:hypothesisI}\eqref{hyp:thexhoel} for the stochastic running costs $\psi$ with $\alpha = 0$ and Hypothesis~\ref{hyp:hypothesisII}.
Further, the derivative of $\obsf$ with respect to $x$ is bounded.
\setcounter{itemfreeze}{\value{enumi}}
\end{enumerate}
For the costs, we use the functional~\eqref{equ:exampleerror}, where 
\begin{enumerate}
\setcounter{enumi}{\value{itemfreeze}}
\item \label{hyp:exmplsigma}
$R, S$ are continuously differentiable matrix valued functions, and $S(t) \geq s \mathbbm{1}$ for some $s > 0$.
\setcounter{itemfreeze}{\value{enumi}}
\end{enumerate}
This implies
\[
\phi(t, x, u) = \frac{1}{2}\trp{\obsf(t, x)}R(t)\obsf(t, x)
+ \frac{1}{2} \trp{u}S(t)u.
\] 
The cost functional corresponds to a minimum energy estimation or weakly constrained 4d--VAR in the geosciences.
By our choice of $\phi$, we can conclude that Hypothesis~\ref{hyp:hypothesisI}\eqref{hyp:thexconvexity} is satisfied with $r = 2, \delta = 0$.
Further
\[
\psi(t, x) = -\trp{\obsf(t, x)}R(t)
\] 
and hence $\psi$ satisfies Hypothesis~\ref{hyp:hypothesisI}\eqref{hyp:thexhoel} with $\kappa \leq 1$ and $\alpha = 0$, and we obtain that
$\gamma(\alpha + \kappa - \frac{1}{2}) + \frac{\beta}{2} \leq \frac{3}{2}$ which is smaller than $r = 2$.
Theorem~\ref{thm:existence} is therefore in force and we obtain that there is an optimal solution to this data assimilation problem for every observation path, and every optimal solution will satisfy $\norm{u}_2 < \infty$.
With regards to the maximum principle, we note that Hypothesis~\ref{hyp:hypothesisII}\eqref{hyp:thmaxdetcost} is satisfied and $\phi$ is strongly convex in $u$.
Further $\DD_2 \phi$ does not depend on $u$ and is continuous in $(t, x)$, whence Hypothesis~\ref{hyp:hypothesisII}\eqref{hyp:thmaxinteg} is satisfied.
Therefore, the maximum principle~\ref{thm:maximumprinciple} applies, and so does Theorem~\ref{thm:hamiltoniansystem}.
The Hamiltonian equations~\eqref{equ:hamiltoniansystem} in the present case read as 
\[
\dot{x} = f(t, x(s)) + g(s, x(s)) u(s)
\]
and (written in coordinates)
\[
\begin{split}
\lambda_m 
& = \sum_j \int_t^T \lambda_j(s) 
\{\DD_2 f(s, x(s)) + \DD_2 g(s, x(s)) u(s)\}_{jm} \: \dd s \\
& \quad + \sum_{jk} \int_t^T \partial_m \obsf_j (s, x(s))
    R_{jk}(s) \{\obsf_k(s, x(s)) \: \dd s - \dd \eta_k(s)\}
\end{split}
\]
with $u(t) = - S^{-1} \trp{g(t, x(t))} \lambda(t)$ and  $x(0) = \xi$.
In this case, it follows directly that the control is continuous and has $q$--variation for any $q > 2$.
The conditions of Theorem~\ref{thm:uniqueness} will be investigated in a future paper, but we conjecture that condition~(b) is satisfied if for example $\obsf$ is linear in $x$ and $f, g$ have continuous second derivatives with respect to $x$.
In this important special case, we obtain the conclusions of Theorem~\ref{thm:uniqueness}; in particular, optimal controls are unique if the value function has a derivative at $\xi$.
As a final note, we discuss the variational problems encountered in~\cite{zeitouni_existence_1988} in connection with the maximum aposteriori (MAP) estimator of trajectories of diffusion processes. 
The conditions (a, b, c) on the state equation~\ref{equ:examplestate} are strengthened to
\begin{enumerate}
\renewcommand{\labelenumi}{(\theenumi')}
\item $f_1$ has bounded and continuous derivatives up to second order with respect to $x$.
\item same as (b).
\item $g$ does not depend on $t$, and $g \trp{g}$ has bounded and continuous derivatives up to third order with respect to $x$.
Further, $0 < c_1 \mathbbm{1} \leq g \trp{g} \leq c_2 \mathbbm{1}$.
\end{enumerate}
Under these conditions, $\Gamma := (g \trp{g})^{-1}$ defines a Riemannian metric, and we write $\sigma$ and $\div$ for the scalar curvature and the divergence, respectively, associated with this metric.
The conditions on $\psi$ and $R$ remain the same. 
In~\cite{zeitouni_existence_1988} it is shown that under certain conditions (which we do not verify here) a MAP estimator for trajectories of the diffusion
\[
\dd x(t) = \tilde{f}(t, x(t)) \dd t + g(x(t)) \: \dd B(t)
\]
with respect to observations of the form
\[
\dd \eta(t) = h(t, x(t)) \dd t  + \rho_t \: \dd B'(t)
\]
is a solution of Problem~VAR with $\psi$
\[
\begin{split}
\psi(t, x) & = -\trp{h(t, x)} R(t), \\
R(t) & = (\rho(t) \trp{\rho(t)})^{-1}, \\
\phi(t, x, u) & = \frac{1}{2} \trp{h(t, x)} R(t) h(t, x)
    + \frac{1}{2} \trp{u} \Gamma(x) u 
    - \div f(t, x)
    + \frac{1}{6} \sigma(x),
\end{split}
\]
and $\tilde{f}$ is related to $f$ through an It\^{o}--Stratonovi\v{c} conversion (see~\cite{zeitouni_existence_1988} for details).
The new terms appearing in the cost function are $\sigma$, $\div f_1$ and $\div f_2$.
The first two are bounded with bounded derivatives, while the third is linear in $x$.
Hence Hypothesis~\ref{hyp:hypothesisII} is again satisfied with $\alpha, \beta, \gamma$ the same constants as before, only $\delta = 1$ now instead of zero as before.
We can thus draw the same conclusions as for the minimum energy estimator, in particular we can conclude the existence of global minimisers.
Note that this has been concluded in~\cite{zeitouni_existence_1988} already, but under much stronger assumptions.
In particular, the boundedness of $\div f$ is needed for that proof to work, meaning that state equations with a quadratic nonlinearity as considered here are not covered.
%
%

%
%
\end{document}